\documentclass[12pt]{amsart}
\usepackage{amsmath}
\usepackage{amsfonts}
\usepackage{amsthm}
\usepackage{amssymb}
\usepackage{amscd}
\usepackage[all]{xy}
\usepackage{enumerate}
\usepackage{hyperref}
\usepackage{comment}


\subjclass[2010]{}

\theoremstyle{plain}
\newtheorem{thm}{Theorem}[section]

\newtheorem{prop}[thm]{Proposition}

\newtheorem{cor}[thm]{Corollary}

\newtheorem{lem}[thm]{Lemma}

\theoremstyle{definition}

\newcommand{\sF}{\mathcal{F}}

\newcommand{\sH}{\mathcal{H}}
\newcommand{\sI}{\mathcal{I}}
\newcommand{\sJ}{\mathcal{J}}
\newcommand{\sK}{\mathcal{K}}
\newcommand{\sL}{\mathcal{L}}
\newcommand{\sN}{\mathcal{N}}
\newcommand{\sM}{\mathcal{M}}
\newcommand{\sO}{\mathcal{O}}

\newcommand{\sS}{\mathcal{S}}

\newcommand{\sX}{\mathcal{X}}
\newcommand{\sY}{\mathcal{Y}}

\newcommand{\mP}{\mathbb{P}}

\newcommand{\mZ}{\mathbb{Z}}

\newcommand{\Pic}{\mathrm{Pic}\,}
\newcommand{\Spec}{\mathrm{Spec}\,}

\newcommand{\Sym}{\mathrm{Sym}}

\usepackage{color}

\def\geq{\geqslant}
\def\leq{\leqslant}

\numberwithin{equation}{section}

\newcommand{\beba}  {\begin{equation}\begin{array}{rcl}}

\newcommand{\eaee}  {\end{array}\end{equation}}

\makeatletter
\def\l@section{\@tocline{1}{0pt}{1pc}{}{}}
\def\l@subsection{\@tocline{2}{0pt}{1pc}{4.6em}{}}
\def\l@subsection{\@tocline{3}{0pt}{1pc}{7.6em}{}}
\renewcommand{\tocsection}[3]{%
  \indentlabel{\@ifnotempty{#2}{\makebox[2.3em][l]{%
    \ignorespaces#1 #2.\hfill}}}#3}
\renewcommand{\tocsubsection}[3]{%
  \indentlabel{\@ifnotempty{#2}{\hspace*{2.3em}\makebox[2.3em][l]{%
    \ignorespaces#1 #2.\hfill}}}#3}
\renewcommand{\tocsubsection}[3]{%
  \indentlabel{\@ifnotempty{#2}{\hspace*{4.6em}\makebox[3em][l]{%
    \ignorespaces#1 #2.\hfill}}}#3}
\makeatother

\title{On deformations of the surfaces of bitangents to smooth quartic surfaces in $\mP^3$}

\author{Ciro Ciliberto}
\address{
Ciro Ciliberto\\Department of Mathematics, \\
Universit\`a di Roma ``Tor Vergata''\\
Via della Ricerca  Scientifica, 00177 Roma\\ Italia
\texttt{cilibert@axp.mat.uniroma2.it   }}

\author{Sandro Verra}
\address{Sandro Verra\\Department of Mathematics \\
Universit\`a di Roma 3\\
Largo San Leonardo Murialdo,  00146 Roma \\ Italia
\texttt{sandro.verra@gmail.com}}

\author{Francesco Zucconi}
\address{Francesco Zucconi\\Department of Mathematics, Computer Science and Physics \\
Universit\`a di Udine\\
33100 Udine\\ Italia
\texttt{francesco.zucconi@dimi.uniud.it}}

\begin{document}

\markboth{}{}
	\label{cast2}
\maketitle

Abstract. {We prove that the surface $S(X)$ of bitangent lines of a general smooth quartic surface $X$ in $\mP^3$  has unobstructed deformations of dimension $20=h^1(S(X), T_{S(X)})$.  In addition, we show that the space of infinitesimal embedded deformations of $X$ injects into the one of $S(X)$. Finally we prove that there is a natural birational map from the 20--dimensional moduli space of (polarised) double coverings of EPW--sextics to the moduli space of  regular surfaces $S$ with $p_g=45$ and $K_S^2=360$ polarised with a very ample line bundle $H$ such that $H^2=40$, $h^0(S, H)=6$: the map sends a double covering of a EPW--sextic in $\mP^5$  to the surface of double points of the EPW--sextic. }

\section{Introduction}

Let $X$ be a smooth quartic surface in $\mP^3$ that does not contain any line. The surface $S(X)$ of its bitangent lines, that naturally sits in the Grassmannian $\mathbb G:=\mathbb G(1,3)$ of lines in $\mP^3$ Pl\"ucker embedded in $\mP^5$, is smooth and irreducible. This surface has been studied by several authors (see, for instance, in time order \cite {W, T, Cl}, and more recently \cite {CZ2, CZ3}). It is known, for example, that $S(X)$ is regular, minimal, with $p_g=45$ and $K_{S(X)}^2=360$. 

A first objective of the present paper is to prove, in \S \ref {sec:natis}, that, for $X$ general, a non--trivial first order embedded deformation of $X$ induces an infinitesimal deformation of $S(X)$ that gives a non--trivial infinitesimal variation of Hodge structures (IVHS), see Theorem \ref {finale}. As a consequence one has that, if $X$ is a general quartic surface in $\mP^3$, the infinitesimal deformations of $S(X)$  induced by the embedded infinitesimal deformations of $X$ inject into $H^1(S(X), T_{S(X)})$ (that has dimension 20, see  \cite [\S 7]{CZ3}), fill up a hyperplane of $H^1(S(X), T_{S(X)})$ and  satisfy the infinitesimal Torelli claim (see Corollary \ref {perkuranishi} and Theorem \ref {nientemale}), whereas the general infinitesimal deformation in $H^1(S(X), T_{S(X)})$ does not satisfy the infinitesimal Torelli claim (see \cite [\S 4]{CZ3}). 

This infinitesimal information has a geometric counterpart as we show in \S \ref {sec:tor}. Consider  the $20$ dimensional  component $\sH$ of the moduli space of hyperkh\"aler 4--folds parametrizing the (polarised) double covers of  EPW--sextics. It is known that $\sH$ contains  points corresponding to the Hilbert squares $X^{[2]}$ of all smooth quartic surfaces $X\subset\mP^3$.  So, if $\sK$ is the moduli space  of (polarised) smooth quartic surfaces in $\mP^3$, that has dimension 19, we can naturally consider $\sK$ as a subvariety of $\sH$.  Then there is a natural map $\varphi: \sH\dasharrow \sM$, where $\sM$ is the moduli space of  regular surfaces $S$ with $p_g=45$ and $K_S^2=360$, polarised with a very ample line bundle $H$ such that $H^2=40$, $h^0(S, H)=6$. The map $\varphi$ sends a double cover of a general EPW--sextic in $\mP^5$  to the surface of double points of the EPW--sextic. In particular, if $X$ is a smooth quartic surface in $\mP^3$ with no line,  $\varphi$ sends $X^{[2]}$ to $S(X)\subset \mathbb G$. So $\sM$ contains a subvariety $\sS$ whose points correspond to the surfaces of bitangents to smooth quartics in $\mP^3$ with no lines and the map $\phi$ induces a map $\varphi': \mathcal K\dasharrow \mathcal S$. We prove in Proposition \ref 
{prop:iso} that  $\varphi'$ induces a bijection between the open set of $\mathcal K$ parametrizing smooth quartics with no lines and $\mathcal S$. This gives another proof of 
the fact that for any smooth quartic surface $X$ with no line the infinitesimal deformations of $S(X)$  induced by the embedded infinitesimal deformations of $X$ inject into $H^1(S(X), T_{S(X)})$. Moreover we prove  that the map $\varphi: \mathcal H\dasharrow \mathcal M$ is birational (see Theorem \ref {thm:bir}). This implies that $\mathcal M$ has dimension $20$ and it is reduced. Section \S \ref {sec:prel} is devoted to collect some preliminary material used in the rest of the paper.

We work over the complex numbers.\medskip

{\bf Acknowledgements:} Ciro Ciliberto is a member of GNSAGA of the Istituto Nazionale di Alta Matematica ``F. Severi".

\section{Preliminaries and notation}\label{sec:prel}
In this section we will introduce concepts and notation that we are going to use, without any further mention, in the rest of the paper.

\subsection{Double covers}
Let $\pi\colon Y\longrightarrow X$ be a double cover of a smooth algebraic variety $X$ branched along a smooth divisor $Z$, so that also $Y$ is smooth. Then. $\pi_*\sO_Y=\sO_X\oplus \sL^\vee$ where $\sL$ is a line bundle such that $\sL^{\otimes 2}=\sO_X(Z)$. Since $\pi$ is finite, 
 we have $ H^i(Y, T_Y)=H^i(X,\pi_* (T_Y))$  for all non--negative integers $i$. If $\pi\colon Y\longrightarrow X$ is \'etale, one has $T_Y\cong \pi^*(T_X)$, hence by the projection formula we have $\pi_*(T_Y)=T_X\oplus (T_X\otimes \sL^\vee)$, hence
\begin{equation}\label{prima decomposizione}
 H^1(Y, T_Y)= H^1(X, T_X) \oplus H^1(X, T_X\otimes \sL).
\end{equation}

\subsection{ Embedded infinitesimal deformations of hypersurfaces}\label{ssec:emb}
Let $X\subset\mP^n$ be a smooth hypersurface of degree $d$ with defining equation $F=0$. Let $B_{s}\subset H^0(\mP^n, \sO_{\mP^n}(d))$ be the Zariski open set of smooth hypersurfaces. Let $\pi\colon\sI\to B_{s}$ be the universal hypersurface. Denote by $\oplus_{m=0}^{\infty}\mathbb C[x_0,\cdots , x_n]_m$ the graded ring $\mathbb C[x_0,\cdots , x_n]$ and by $J_F=\oplus_{m=0}^{\infty}J^m$ the Jacobian ideal, that is the ideal generated by the partial derivatives $\partial F/\partial x_i$, $i=0,\ldots,n$. We denote by $R_F=\oplus_{m=0}^{\infty}R^m_F$ the \emph{Jacobian ring} of $X$, i.e., the graded ring $\mathbb C[x_0,\cdots , x_n]/J$. Clearly $\mathbb C[x_0,\cdots , x_n]_d\cong T_{B_s,[X]}$. If $n\geq 2$, we have the \emph{Kodaira--Spencer map}
$$
ks\colon T_{B_s,[X]}\longrightarrow H^1(X, T_X).
$$
It is known that the kernel of this map is $J^d_F$, namely $ks$ factors through an injective map
$$
\kappa\sigma: R^d_F\longrightarrow H^1(X, T_X).
$$
Moreover $ks$ is surjective if either $n\geq 4$ or $(n,d)\neq (3,4)$, in which case it has corank 1 (for all this  see \cite[Lemma 6.15, Remark 6.16]{Vo2}).

\subsection{Bitangent lines to a quartic surface in $\mP^3$} 
Let $X\subset \mP^3$ be a smooth quartic surface that contains no line. A line $l\subset\mP^3$ is a \emph{bitangent line} to $X$ if it is tangent to $X$ at each point of $l\cap X$, i.e., if the divisor cut out by $X$ on $l$ is of the form $2Z$, where $Z$ is a divisor of degree 2 on $l$. The locus
$$
S(X):=\{ [l]\in\mathbb G(1,3)\mid l\, {\rm{is}}\,  {\rm{bitangent}}\, {\rm{to}}\, X\}
$$
is a smooth surface in the grassmannian $\mathbb G(1,3)\subset \mP^5$ of lines in $\mP^3$ (see \cite [Lemma (1.1)]{W}) called 
the 	\emph { surface of bitangents to $X$}. Various properties of $S(X)$ are well known. For our purposes here we recall that $S(X)\subset \mP^5$ has degree 40, it is minimal, regular,  its geometric genus is 45 and $K^2_{S(X)}= 360$ (see 
\cite [\S 3]{CZ3}). 

\subsection{The quartic double solid}\label{ssec:4ic}
Given $X$, we can consider the double cover $\pi: \sX\longrightarrow \mP^3$ branched along $X$, that is called a \emph{quartic double solid}. This is a smooth Fano threefold of index 2, with $\Pic(\sX)\cong \mZ$, generated by $L:=\pi^*(\sO_{\mP^3}(1))$ (see \cite [p. 8]{W}). 
One defines \emph{lines} on $\sX$ the (smooth) curves $C$ on $\sX$ such that $C\cdot L=1$. Any such a curve maps to a bitangent line to $X$ and if $l$ is a bitangent line to $X$, then $\pi^{-1}(l)$ is the union of two distinct lines of $\sX$, intersecting along a 0--dimensional scheme of length 2. If $S_X$ denotes the Hilbert scheme of lines of $\sX$, then $S_X$ is irreducible and $\pi$ induces a morphism $f: S_X\longrightarrow S(X)$, that is \'etale of degree $2$, so that there is a 2--torsion class $\sigma\in {\rm NS}(S(X))$, such that  $f_*(\mathcal O_{S_X})=\mathcal O_{S(X)}\oplus \sigma$ (see \cite [Corollary (1.3)]{W}). Hence $S_X$ is also a smooth surface and it is known that it is irregular, with irregularity $10$ and geometric genus 101 (see again 
\cite [\S 3]{CZ3}). It is known that the intermediate jacobian $J(\sX)$
of $\sX$ is isomorphic to the Albanese variety of $S_X$ (see \cite [Thm. (4.1)]{W}). Hence we have the \emph{Albanese morphism} $a_{S_X}: S_X\longrightarrow {\rm Alb}(S_X)\cong J(\sX)$. By \cite [Prop. (2.13)]{W} this is an immersion. It is also known  (see \cite [Diagram (2.14)]{W})  that there is a natural isomorphism
\begin{equation}\label{eq:iso}
T_{J(\sX),0}^\vee\cong H^0(\mP^3, \sO_{\mP^3}(2)).
\end{equation}
In what follows we will think of $\mP^3$ as $\mP(V^\vee)$, where $V$ is a complex vector space of dimension 4. Then the isomorphism \eqref{eq:iso} reads as
\begin{equation}\label{eq:iso1}
T_{J(\sX),0}\cong \Sym^2(V)^\vee
\end{equation}

\subsection{Some natural isomorphisms}

Since we have the \'etale  double cover $f: S_X\longrightarrow S(X)$, accordingly there is an involution $\iota: S_X\longrightarrow S_X$. If $\sF$ is a sheaf on $S_X$ such that $\iota^*\sF=\sF$, then for any cohomology space $H^i(S_X,\sF)$ we denote by $H^i(S_X,\sF)^+$ [resp. by $H^i(S_X,\sF)^-$] the invariant [resp. the antiinvariant] subspace of $H^i(S_X,\sF)$ by the involution $\iota$.

\begin{prop}\label{isocoomologia} There are  natural isomorphisms
\begin{enumerate}[(i)]
\item $ \bigwedge^2 H^0(S_X,\Omega^1_{S_X})\cong H^0(S_X,\Omega^2_{S_X})^{+}$;
\item $ H^0(S_X,\Omega^1_{S_X})\otimes_{\mathbb C}H^1(S_X,\sO_{S_X})\cong H^1(S_X,\Omega^1_{S_X})^{+} $.
\end{enumerate}
\end{prop}
\begin{proof} See  \cite[formulae (3.61), (3.62)]{W}. \end{proof}
\begin{cor}\label{coomologiadellaalbanese}
One has
\begin{enumerate}[(i)]
\item $H^{0}({\rm{Alb}}(S_X),\Omega^2_{{\rm{Alb}}(S_X)})\simeq H^0(S_X,\Omega^2_{S_X})^{+}$;
\item $ H^{1}({\rm{Alb}}(S_X),\Omega^1_{{\rm{Alb}}(S_X)})\simeq H^0(S_X,\Omega^1_{S_X})\otimes_{\mathbb C}H^1(S_X,\sO_{S_X}).$
\end{enumerate}
\end{cor}
\begin{proof} Part (i) follows immediately by part (i) of Proposition \ref{isocoomologia}. As for part (ii), it follows by  a general property of complex  tori. Indeed, if $A$ is a complex torus of dimension $n$, then one  has the obvious map
$$
H^0(A,\Omega^1_A)\otimes H^1( A,\sO_A)\longrightarrow H^1(A, \Omega^1_A)
$$
that boils down to the map
$$
H^{1,0}(A)\otimes H^{0,1}(A)\longrightarrow H^{1,1}(A), \quad \alpha\otimes \beta \mapsto \alpha\wedge \beta
$$
which is an isomorphism. Since $H^0(S_X,\Omega^1_{S_X})\cong H^0({\rm{Alb}}(S_X),\Omega^1_{{\rm{Alb}}(S_X)})$ and $H^1(S_X,\sO_{S_X})\cong H^0({\rm{Alb}}(S_X),\sO_{{\rm{Alb}}(S_X)})$, the assertion follows.
\end{proof}

\begin{cor}\label{coomologiadellasuperficiedellebitangenti} There are natural identifications
$$H^0(S(X),\Omega^2_{S(X)})\simeq H^{0}({\rm{Alb}}(S_X),\Omega^2_{{\rm{Alb}}(S_X)})$$
$$H^1(S(X),\Omega^1_{S(X)})\simeq H^{1}({\rm{Alb}}(S_X),\Omega^1_{{\rm{Alb}}(S_X)}).$$
 \end{cor}
\begin{proof} One has $H^0(S(X),\Omega^2_{S(X)})=H^0(S_X,\Omega^2_{S_X})^{+}$ and  $H^1(S(X),\Omega^1_{S(X)})=H^1(S_X,\Omega^1_{S_X})^{+}$. Hence the claim follows by  Proposition \ref {isocoomologia} and  Corollary \ref{coomologiadellaalbanese}. \end{proof}

\section{Infinitesimal deformations}\label{sec:natis}

\subsection{The equivariant homorphism} Let $X$ be a smooth quartic surface in $\mP^3=\mP(V^\vee)$, not containing any line, with equation $F=0$, where $F\in \Sym^4(V)$. 
Fix $G={\rm{Sym}}^4(V)$. This determines a first order infinitesimal  embedded deformation $F+\epsilon G=0$, over the ring of dual numbers $\mathbb C[\epsilon]$. Accordingly, this gives an element
 $\xi_G\in H^1(X, T_X)$ and it induces an infinitesimal deformation of $S(X)$, $\sX$, $S_X$, $J(\sX)\cong {\rm{Alb}}({S_X})$. 
 
 We will denote by $\xi^{  {\rm{Alb}}({S_X})  }_{G}\in H^1({\rm{Alb}}_{S_X}  , T_{{\rm{Alb}}(S_X)})\cong H^1(J(\sX), T_{J(\sX)})$ the first order infinitesimal deformation induced on ${\rm{Alb}}(S_X)\cong J(\sX)$, that is a principally polarised abelian variety (ppav). Hence the space of first order infinitesimal deformations of $J(\sX)$ as a ppav is 
 $$
 \Sym^2 (T_{J(\sX),0})\cong \Sym^2(\Sym^2(V)^\vee)=\Sym^4(V)^\vee,
 $$
  (remember \eqref {eq:iso1})  and therefore $\xi^{  {\rm{Alb}}({S_X})  }_{G}\in \Sym^4(V)^\vee$. 
 We have thus defined a linear map
 $$
 \Phi_X\colon {\rm{Sym}}^4(V)\longrightarrow {\rm{Sym}}^4(V)^\vee,\, \, G\mapsto {\xi^{ {\rm{Alb}}({S_X} ) }_{G}}.
 $$
 
 \begin{prop}\label{livello abeliano} Let $X$ be a general quartic surface in $\mP^3$ with equation $F=0$. Then the map $\Phi_X$ 
 is a ${\rm PGL}(V)$--equivariant homomorphism which factors through $R_F^4$ and the induced homomorphism $\Phi_{F}\colon R_F^4\to{\rm{Sym}}^4(V)^\vee$ is injective.
 \end{prop}
 \begin{proof} That $\Phi_X$  is ${\rm PGL}(V)$--equivariant is obvious. If $G\in J^4_F$, we know that $\xi_G=0$ (see \S \ref {ssec:emb}), hence also $\xi^{ {\rm{Alb}}({S_X} ) }_{G}$ is trivial. Thus $J^4_F\subset{\rm{Ker}}(\Phi_X)$, hence $\Phi_X$  factors through $R_F^4$. 
 
  Let us now prove the injectivity of $\Phi_{F}$. Let $[G_1],[G_2]\in R_F^4$ be such that $[G_1]\neq [G_2]$. Suppose that $\Phi_{F}([G_1])=\Phi_{F}([G_2])$, i.e.,  ${\xi^{  {\rm{Alb}}(S_X) }_{G_1} }={\xi^{ {\rm{Alb}}(S_X)}_{G_2} }$. Let $\pi_1\colon \sJ_{G_1}\to \Spec(\mathbb C[\epsilon])$, $\pi_2\colon \sJ_{G_2}\to \Spec(\mathbb C[\epsilon])$ be the corresponding infinitesimal first order deformations  of intermediate jacobians of the quartic double solids corresponding to the quartics with equations $\{ X_{1,\epsilon}=F+\epsilon G_1\}$, $\{X_{2,\epsilon}=F+\epsilon G_2\}$ respectively. Then $\pi_1$ and $\pi_2$ are the same  infinitesimal deformation for the intermediate jacobian $J(\sX)$. Since infinitesimal deformation of principally polarised abelian varieties are unobstructed, we can prolong $\pi_1\colon \sJ_{G_1}\to \Spec(\mathbb C[\epsilon])$, $\pi_2\colon \sJ_{G_2}\to \Spec(\mathbb C[\epsilon])$ to families 
 $\pi'_1\colon \sY_{G_1}\to \mathbb D$, $\pi'_2\colon \sY_{G_2}\to \mathbb D$ over a disk $\mathbb D$ given by the intermediate jacobians of the quartic double solids given by  $\{ X_{1,t}=F+tG_1\}_{t\in\mathbb D}$, $\{X_{2,t}=F+tG_2\}_{t\in\mathbb D}$ respectively, and we can assume that  they give the same deformation of  $J(\sX)$. Thus, for generic $t\in\mathbb D$, one has that $J(\sX_{1,t})$ is isomorphic to $J(\sX_{2,t})$ as principally polarised abelian varieties. By \cite[Th\'eor\`eme (9.2)]{D} we have that $X_{1,t}$ is isomorphic to $X_{2,t}$. This implies that $[G_1]=[G_2]$, a contradiction.
\end{proof}

 \subsection{IVHS of the Albanese variety}
 The space of infinitesimal deformations of the Abelian variety ${\rm{Alb}}(S_X)$ as a complex torus is $H^1({\rm{Alb}}(S_X),T_{{\rm{Alb}}(S_X)})$. Since  $X\subset \mP^3=\mathbb P(V^\vee)$,  and the tangent space of $J(\sX)\cong {\rm{Alb}}(S_X)$ at $0$ is ${\rm{Sym}}^2(V)^\vee$ (see \eqref {eq:iso1}), then 
 $T_{{\rm{Alb}}(S_X)} ={\rm{Sym}}^2(V)^{\vee}\otimes_{\mathbb C}\sO_{{\rm{Alb}}(S_X)}$. Hence  
 $$
H^1({\rm{Alb}}(S_X),T_{{\rm{Alb}}(S_X)})={\rm{Sym}}^2(V)^{\vee}\otimes_{\mathbb C} H^1({\rm{Alb}}(S_X),\sO_{{\rm{Alb}}(S_X)})={\rm{Sym}}^2(V)^{\vee}\otimes_{\mathbb C}{\rm{Sym}}^2(V)^{\vee}
 $$
 because $H^1({\rm{Alb}}(S_X),\sO_{{\rm{Alb}}(S_X)})\cong {\rm{Sym}}^2(V)^{\vee}$. 
The space of first order infinitesimal deformations of the ppav  $({\rm{Alb}}(S_X), \theta_{{\rm{Alb}}(S_X)})$ is
 $$
{\rm{Sym}}^2(T_{{\rm{Alb}}(S_X),0})={\rm{Sym}}^2({\rm{Sym}}^2(V)^\vee)= {\rm{Sym}}^4(V)^{\vee}
$$
To compute the differential of the period map is equivalent to compute for every non--negative integer $q$, the  map induced by the cup product
$$
 H^1({\rm{Alb}}(S_X),T_{{\rm{Alb}}(S_X)})\to {\rm{Hom}}(H^{q}({\rm{Alb}}(S_X),\Omega^p_{{\rm{Alb}}(S_X)}), H^{q+1}({\rm{Alb}}(S_X),\Omega^{p-1}_{{\rm{Alb}}(S_X)})).
 $$
For abelian varieties the cup product reduces to the homomorphism induced by the contraction. In particular the homomorphism
 $$
 H^1({\rm{Alb}}(S_X),T_{{\rm{Alb}}(S_X)})\to {\rm{Hom}}(H^{0}({\rm{Alb}}(S_X),\Omega^2_{{\rm{Alb}}(S_X)}), H^{1}({\rm{Alb}}(S_X),\Omega^{1}_{{\rm{Alb}}(S_X)}))
 $$
 translates into the  contraction homomorphism
 $$
 {\rm{Sym}}^2(V)^{\vee}\otimes_{\mathbb C}{\rm{Sym}}^2(V)^{\vee}\to {\rm{Hom}}( \bigwedge^2{\rm{Sym}}^2(V)\to  {\rm{Sym}}^2(V)\otimes {\rm{Sym}}^2(V)^{\vee})
 $$
 defined over indecomposable vectors as follows
 $$
 (\alpha\otimes \beta)(\phi\wedge \psi)=\alpha(\phi)\cdot\psi\otimes \beta-\alpha(\psi)\cdot\phi\otimes \beta
 $$
 where $\alpha,\beta\in {\rm{Sym}}^2(V)^{\vee}$ and $\phi,\psi\in {\rm{Sym}}^2(V)$.
 
 In general, for any polarised abelian variety $(A, \theta_{A})$, one has:
 
 \begin{prop}\label{abelianocaso} The natural homomorphism from the space  of infinitesimal deformations of a polarised Abelian variety $(A, \theta_{A})$ of dimension $n\geq 2$ to the $(2,0)$--IVHS is injective.
  \end{prop}
  \begin{proof} Set $W=H^0(A, \Omega^1_A)$. The space  of infinitesimal deformations of the polarised Abelian variety $(A, \theta_{A})$ can be identified with  the space of symmetric homomorphism ${\rm{Hom}}^{s}(W,W^\vee)$ where $\phi\in {\rm{Hom}}^{s}(W,W^\vee)$ if and only if $\langle \phi(w_1),w_2\rangle=\langle \phi(w_2),w_1\rangle$. To compute the $(2,0)$--infinitesimal variation of Hodge structure induced by $\phi$ means to compute the homomorphism $\Phi\in {\rm{Hom}}(H^0(A,\Omega^2_A),H^1(A,\Omega_A^1))$ 
  $$
  \Phi\colon \bigwedge^2W\to W^\vee\otimes W
  $$
  given by $\Phi(w_1\wedge w_2)=\phi(w_1)\otimes w_2-\phi(w_2)\otimes w_1$. It is immediate that if $\phi\neq 0$ then $\Phi\neq 0$.
  \end{proof}
 
 In our special case the natural homomorphism from the space  of infinitesimal deformations of $({\rm{Alb}}(S_X), \theta_{{\rm{Alb}}(S_X)})$ to the $(2,0)$--IVHS  can be written as the  contraction injective homomorphism
 \begin{equation}\label{ilsigmaquattro}
\tau\colon {\rm{Sym}}^4(V)^\vee \to {\rm{Hom}}_{\mathbb C}( \bigwedge^2{\rm{Sym}}^2(V),{\rm{Sym}}^2(V)^\vee\otimes {\rm{Sym}}^2(V))
 \end{equation}
 which to $\phi\in{\rm{Sym}}^4(V)^\vee$ associates $\tau(\phi)\colon\bigwedge^2{\rm{Sym}}^2(V)\to {\rm{Sym}}^2(V)^\vee\otimes {\rm{Sym}}^2(V)$ such that 
 $$
 \tau(\phi)(q_1\wedge q_2)=\phi(q_1)\wedge q_2-\phi(q_2)\wedge q_1.
 $$


 \subsection{IVHS for the infinitesimal deformations coming from embedded deformations}  We keep the notation introduced above.
Let $G\in H^0(X,\sO_{\mP^3}(4))={\rm{Sym}}^4(V)$ and let $\xi_G\in H^1(X, T_X)$ be the first order  infinitesimal embedded deformation of $X$ given by $F+\epsilon G=0$, that we assume to be non--trivial. Let $\xi^{S(X)}_G\in H^1(S(X), T_{S(X)})$ be the induced first order  infinitesimal deformation of $S(X)$.
The double cover $f\colon S_X\to S$ is \'etale and we have $f_*\sO_{S_X}=\sO_{S(X)}\oplus \sigma$, where $\sigma$ is non--trivial such that $\sigma^{\otimes 2}=\sO_{S(X)}$. By 
 \eqref{prima decomposizione} we have
\begin{equation}\label{seconda decomposizione}
H^1(S_X, T_{S_X})= H^1(S(X), T_{S(X)})\oplus H^1(S(X), T_{S(X)}\otimes_{\sO_S}\sigma)
\end{equation}
By compatibility, we have that $\xi^{S(X)}_{G}=\pi^{+}(\xi^{S_X}_{G})$ where $\pi^{+}\colon H^1(S_X, T_{S_{X}})\to H^1(S(X), T_{S(X)})$ is the projection given by the above splitting of $H^1(S_X, T_{S_{X}})$ into invariant and anti--invariant subspaces. 

We stress that $G$ also induces the first order infinitesimal deformations $\xi^{S_X}_{G}\in H^1(S_X, T_{S_X})$ and $\xi^{  {\rm{Alb}}({S_X})  }_{G}\in H^1({\rm{Alb}}({S_X}), T_{{\rm{Alb}}({S_X})})$ of $S_X$ and of ${\rm{Alb}}({S_X})$. Recall that the Albanese morphism $a_{S_X}: S_X\longrightarrow {\rm{Alb}}({S_X})$ is an immersion.  Consider then the sequence 
$$
0\longrightarrow T_{S_{X}} \stackrel{d_{a_{S_X}}} {\longrightarrow} a_{S_X}^*T_{{\rm{Alb}}({S_X})}\longrightarrow \sN_{a_{S_X}}\longrightarrow 0
$$
which induces a homomorphism $d_{{a}_{{S_X}}}\colon H^1(S_X, T_{S_{X}})\longrightarrow H^1(S_X, a_{S_X}^{*}T_{{\rm{Alb}}({S_X})})$. We set $\eta:=d_{a_{S_X}}(\xi^{S_X}_{G})$. We have also  the natural pull--back homomorphism
$$
a_{S_X}^{*} \colon H^1({\rm{Alb}}({S_X}), T_{{\rm{Alb}}({S_X})})\to H^1(S_X, a_{S_X}^{*}T_{{\rm{Alb}}{(S_X})})
$$
and by  compatibility we have 
$$a_{S_X}^{*} (\xi^{  {\rm{Alb}}({S_X})  }_{G})=\eta=d_{a_{S_X}}(\xi^{S_X}_{G}).$$ 
More precisely, since $\xi^{  {\rm{Alb}}({S_X})  }_{G}$ is an infinitesimal deformation of the ppav $({\rm{Alb}}(S_X), \theta_{{\rm{Alb}}(S_X)})$, then we may view $\xi^{  {\rm{Alb}}({S_X})  }_{G}$ as an element of ${\rm{Sym}}^4(V)^\vee$. 


The compatibility also implies that the following diagram
	\begin{equation}\label{dallalbanese}
		\xymatrix{
		 H^{0}({\rm{Alb}}(S_X),\Omega^2_{{\rm{Alb}}(S_X)})\ar^{ a_{S_X}^{*} }[d] \ar^{
	{\xi^{{\rm{Alb}}({S_X})}_{G} }}[r]  & H^{1}({\rm{Alb}}(S_X),\Omega^1_{{\rm{Alb}}(S_X)}) \ar^{  a_{S_X}^{*}}[d]\\
			 H^0(S_X,\Omega^2_{S_X})\ar^{{\xi^{S_X}_{G} }}[r]& H^1(S_X,\Omega^1_{S_X})\\
		}
	\end{equation}
	commutes, where, by abuse of notation, ${\xi^{{\rm{Alb}}({S_X})}_{G} }$ and ${\xi^{S_X}_{G} }$ denote the cup products by $\xi^{{\rm{Alb}}({S_X})}_{G}$ and $\xi^{S_X}_{G}$ respectively. 

 \begin{thm}\label{finale} If $\xi_G\in H^1(X, T_X)$ is a non--trivial first order embedded deformation, then the infinitesimal deformation $\xi^{S(X)}_{G}$ gives a non--trivial IVHS.
 \end{thm}
 \begin{proof} By Proposition \ref{livello abeliano} the map
 $$ \Phi\colon {\rm{Sym}}^4(V)\to {\rm{Sym}}^4(V)^\vee,\quad  G\mapsto {\xi^{  {\rm{Alb}}(S_X)}_{G} }$$ 
 factors through $R_4$ and $\Phi_{F}\colon R_4\to{\rm{Sym}}^4(V)^\vee$ is injective. Hence the homomorphism 
 $$\tau\circ \Phi_{F}: R_4\to {\rm{Hom}}_{\mathbb C}( \bigwedge^2{\rm{Sym}}^2(V),{\rm{Sym}}^2(V)\otimes {\rm{Sym}}^2(V)^{\vee})$$ 
 is also injective by Proposition \ref {abelianocaso} (recall that $\tau$ is defined in\eqref {ilsigmaquattro}). By compatibility, by Corollary \ref{coomologiadellaalbanese} and  Corollary \ref{coomologiadellasuperficiedellebitangenti} the commutative diagram (\ref{dallalbanese}) can be interpretated as the following one
\begin{equation}\label{dallalbanesedue}
		\xymatrix{
		 H^{0}(S(X),\Omega^2_{S(X)})\ar^{  }[d] \ar^{{\xi^{S(X)}_{G} }}[r]  & H^{1}(S(X),\Omega^1_{S(X)}) \ar^{  }[d]\\
			 H^0(S_X,\Omega^2_{S_X})\ar^{{\xi^{S_X}_{G} }}[r]& H^1(S_X,\Omega^1_{S_X})\\
		}
	\end{equation}
where the vertical arrows are injective since they are induced by the identifications of Corollary \ref{coomologiadellasuperficiedellebitangenti}, and the horizontal rows are cup products as usual. Finally the first row of diagram \eqref {dallalbanese}, that identifies with the first row of diagram \eqref{dallalbanesedue},  computes the $(2,0)$--IVHS of $({\rm{Alb}}(S_X), \theta_{{\rm{Alb}}(S_X)})$ and by  Proposition 
\ref{abelianocaso} it is non trivial.\end{proof}

We recall that if $S$ is a smooth, projective, irreducible algebraic surface, one says that a non--trivial  infinitesimal deformation $\xi\in H^1(S,T_S)$ of $S$ \emph{satisfies the infinitesimal Torelli claim} if at least one of the two cup product maps $\xi^i\colon H^0(S, \Omega_S^i)\to H^1(S, \Omega_S^{i-1})$ associated to $\xi$, for $i=1,2$, is not trivial. The surface $S$ is said to \emph{satisfies the infinitesimal Torelli claim} if  any non--trivial  infinitesimal deformation $\xi\in H^1(S,T_S)$ satisfies the infinitesimal Torelli claim.

Then we have:

\begin{cor}\label{perkuranishi} If $X$ is a general quartic surface in $\mP^3$, the infinitesimal deformations of $S(X)$  induced by the embedded infinitesimal deformations of $X$ satisfy the infinitesimal Torelli claim. Moreover they fill up a hyperplane of $H^1(S(X), T_{S(X)})$, that has dimension 20.
\end{cor}

\begin{proof} The first assertion is clear. The second follows from the fact that the embedded infinitesimal deformations of $X$ form a vector space of dimension 19, whereas $h^1(S(X), T_{S(X)})=20$ (see \cite [\S 7]{CZ3}). \end{proof}

To sum up one has the following:
\begin{thm}\label{nientemale} Let $X\subset\mP^3$ be a general quartic surface and let $S(X)$ be the surface of its bitangent lines. The space $H^1(S(X), T_{S(X)})$ of infinitesimal deformations of $S(X)$ has dimension $20$. The $19$--dimensional space of infinitesimal embedded deformations of $X$ injects into $H^1(S(X), T_{S(X)})$ and any element of its image satisfies the infinitesimal Torelli claim. The general element of  $H^1(S(X), T_{S(X)})$ does not satisfy the infinitesimal Torelli claim. 
\end{thm}
\begin{proof} It follows straightly by Theorem \ref{finale} and by \cite[Main Theorem]{CZ3}.
\end{proof}

\section{A birational map between certain moduli spaces}\label{sec:tor}

In this section we keep the notation introduced above. We will consider the following moduli spaces:
\begin{enumerate}
\item [$\bullet$] the moduli space $\sK$ of (polarised) smooth quartic surfaces in $\mP^3$, that has dimension 19;
\item [$\bullet$]   the irreducible component   $\sH$ of the moduli space of hyperkh\"aler 4--folds containing the (polarised) double EPW--sextic (see for instance \cite {F1,F2}).
\end{enumerate}

 It is known that $\sH$ contains  points corresponding to the Hilbert squares $X^{[2]}$ of all smooth  quartic surfaces $X\subset\mP^3$.  Hence one can consider in a natural way $\sK$ as a subvariety of $\sH$.  
 
 It is also known that $\sH$ has dimension 20 and its general point $V$ has an involution $\iota_V$ that specializes to the \emph{Beauville involution} when $V$ specializes to the Hilbert square $X^{[2]}$ of a smooth quartic surface $X\subset\mP^3$ not containing a line. The Beauville involution $\iota_X$ of $X^{[2]}$ sends a scheme $\eta\in X^{[2]}$ to the scheme $\eta'\in X^{[2]}$ such that the union of the two schemes $\eta$ and $\eta'$ are cut out on $X$ by the line spanned by $\eta$. Note that the fixed points of $\iota_X$ fill up a surface that is isomorphic to the surface $S(X)$ of bitangents to $X$. 

The elements $V\in \sH$ are endowed with a polarization $H_V$, such that $H_V^4=12$  that determines a morphism $\phi_V: V\longrightarrow \Phi_V\subset \mP^5$ that,  in general,   is a 2:1 cover of an EPW--sextic $\Phi_V$. When $V$ specializes to the Hilbert square $X^{[2]}$ of a smooth quartic surface $X\subset\mP^3$ non containing a line, then $\phi_V$ specializes to the map $\phi_X$ that sends a scheme $\eta\in X^{[2]}$ to the line $r_\eta$ spanned by $\eta$, as a point of the Grassmannian $\mathbb G=\mathbb G(1,3)$ that is a smooth quadric in $\mP^5$. The double cover 
$\phi_V: V\longrightarrow \Phi_V$ arises as follows.  For a general $V\in \sH$, given two points $x,y\in V$ one has $\phi_V(x)=\phi_V(y)$ if and only if $\iota_V(x)=y$. Hence the EPW--sextic $\Phi_V$ coincides with the quotient $V/\iota_V$.   The fixed points of the involution $\iota_V$ fill up a  smooth  regular surface $S(V)$ with $p_g=45$ and $K_{S(V)}^2=360$ that specializes to $S(X)$ when $V$ specializes to $ X^{[2]}$. Hence the EPW--sextic  $\Phi_V$ is smooth off a surface of double points isomorphic to  $S(V)$ over which the double cover $V\to \Phi_V$ is ramified (see, for instance, \cite{F1,F2}). When $V$ specializes to the Hilbert square $X^{[2]}$ of a smooth quartic surface $X\subset\mP^3$ non containing a line, then the EPW--sextic $\Phi_V$ specializes to the Grassmannian $\mathbb G$, containing $S(X)$, counted with multiplicity 3. 

 We will consider also the component $\sM$ of the moduli space of smooth regular surfaces with $p_g=45$ and $K^2=360$ containing the surfaces of bitangents to smooth quartics in $\mP^3$ with no  lines, polarised  by the line bundle that maps them as smooth surfaces of degree 40 in $\mathbb G\subset \mP^5$.  Then $\sM$ contains a subvariety $\sS$ whose points correspond to the surfaces of bitangents to smooth quartics in $\mP^3$ with no  lines,  embedded in the Grassmannian $\mathbb G$.

We will need a fact, probably well known to the experts, but for which we did not find any appropriate reference.

\begin{lem}\label{lem:notin} For general $V\in \sH$ the double surface $S(V)$ of the EPW--sextic $\Phi_V$ in $\mP^5$ does not lie on any smooth quadric.
\end{lem}

\begin{proof} We argue by contradiction. Suppose $S(V)$ sits on a smooth quadric $\mathbb Q$ that intersects the EPW--sextic $\Phi_V$ in a 3-fold $W$ of degree 12, singular along $S(V)$. The 3--fold $W$ may be reducible and even non--reduced. However we denote by $W'$ the union of the components of $W$ that contain $S(V)$. 

Let us  make a section with a general subspace $\Pi$ of dimension 3. Then $\Pi$ intersects the EPW--sextic $\Phi_V$ in a sextic surface $\Phi$ with nodes only at the set $Z$ of the 40 intersection points  of $\Pi$ with $S(V)$, it intersects $\mathbb Q$ in a smooth quadric $Q\cong \mP^1\times \mP^1$, it intersects $W'$ in a curve $C$ on $Q$, that is singular at the set $Z$. 

Assume first that $C$ is reduced. Then $C$ would be a curve of type $(a,b)$ on $Q$, with $a,b\leq 6$, and therefore the maximum number of singular points that $C$ could have is $ab\leq 36$, a contradiction, because $C$ must be singular along the 40 points of $Z$. 

So $C$ is non--reduced. The 3--fold $W'$ is contained in $\mathbb Q$, and since the Picard group of $\mathbb Q$ is $\mathbb Z$ generated by $\mathcal O_\mathbb Q(1)$, then a non--reduced component of $W'$ can: (i) either be the complete intersection of $\mathbb Q$ with a hyperplane, (ii) or the complete intersection of $\mathbb Q$ with a quadric, (iii) or the complete intersection of $\mathbb Q$ with a cubic. Since all  components of $W'$  contain $S(V)$,  case (i) cannot happen, because $S(V)$ is non--degenerate in $\mathbb P^5$. So we have to discuss cases (ii) and (iii). Suppose we are in case (ii).  Let $W''$ be an irreducible component of $W'$ that is cut out on $\mathbb Q$ by a quadric, and contains $S(V)$. Then $W''$ intersects $\Pi$ along an irreducible curve $C'$ of type $(2,2)$ on $Q$,  that is a non--reduced component of $C$, and contains the 40 points of $Z$ because $W''$ contains $S(V)$.The sextic $\Phi$ is tangent to $Q$ along $C'$. Assume that $\Phi$ is simply tangent to $Q$ along $C'$ (the argument is similar if $\Phi$ is doubly tangent to $Q$ along $C'$). Then we can consider the  linear system $\mathcal L$ of sextic surfaces in $\mP^3$ that are tangent to $Q$ along $C'$. The general surface $S\in \mathcal L$ has a number $d$ of singular points along $C'$, and if a surface $S\in \mathcal L$ has more that $d$ singular points along $C'$, then it must be singular all along $C'$. To compute the number $d$ let us consider the surfaces $S\in \mathcal L$ of the form $Q+T$, where $T$ is a general quartic surface. The number of singular points of such surfaces along $C'$ is clearly $d=4\deg(C')=16$. On the other hand we have in $\mathcal L$ the surface $\Phi$ that has 40 nodes along $C'$, and this gives a contradiction. The argument is completely similar in case (iii) and we can leave it to the reader.
\end{proof}

We have a natural rational map
$$
\varphi: \mathcal H\dasharrow \mathcal M
$$
that takes the class of the general polarised $(V,H_V)\in \mathcal H$ to the class of the polarised surface $(S(V), H_{S(V)})\in \mathcal M$. This map restricts to the map
$$
\varphi': \mathcal K\dasharrow \mathcal S\subseteq \mathcal M
$$
that takes the class of a (polarised) quartic $X\in \mathcal K$ with no line to the class of the surface $S(X)\in \mathcal S$ of bitangents (with the polarization that embeds it  in $\mathbb G$). 

We notice that Theorem \ref {nientemale} implies that the differential of the map $\varphi'$ at the general point of $\mathcal K$ is injective. Actually we can say more:

\begin{prop}\label{prop:iso} The map $\varphi': \mathcal K\dasharrow \mathcal S$ induces a bijection between the open set of $\mathcal K$ parametrizing quartics with no lines and $\mathcal S$. 
\end{prop}

\begin{proof} Let $X$ be a smooth quartic surface in $\mP^3$ not containing any line. Then $S(X)$ is a smooth surface sitting in $\mathbb G\subset \mP^5$, the Grassmannian of  lines in $\mP^3$. We need to show that there is no other quartic $X'$ distinct from $X$ such that $S(X')=S(X)$. To see this, note that $S(X)$ is a congruence of lines, and as such, it has a \emph{focal locus} $Z(X)$ that is described as follows. For the general line $l$ in $S(X)$, one considers the \emph{focal scheme} of the congruence $S(X)$ on $l$, that consists of a 0--dimensional subscheme of $l$ of lenght 2  (for the notion of foci see \cite  {CS,green, Fa2}). Then $Z(X)$ is swept out by these focal schemes as $l$ varies in $S(X)$.  Now it is a classical fact that for the general $l$ in $S(X)$, the focal scheme of the congruence $S(X)$ on $l$ equals the pairs of contact points of the bitangent $l$ with $X$ (see \cite[Livre IV, Chap. I]{Dar}),
hence $Z(X)=X$. In conclusion $X$ can be reconstructed from $S(X)$ as the focal locus of $S(X)$, and the assertion follows.\end{proof}

Proposition \ref {prop:iso} implies that $\dim(\mathcal S)=\dim (\mathcal K)=19$, hence $\dim (\mathcal M)\geq 19$. Recall that if $X$ is a general quartic surface in $\mP^3$, then $H^1(S(X),T_{S(X)})=20$ (see \cite [\S 7]{CZ3}). Hence $\dim (\mathcal M)\leq 20$. 
Building on Proposition \ref  {prop:iso} we can prove that actually $\dim (\mathcal M)=20$ (and therefore it is reduced) and the following holds:

\begin{thm}\label{thm:bir} The map $\varphi: \mathcal H\dasharrow \mathcal M$ is birational.
\end{thm}

\begin{proof} Lemma \ref {lem:notin} implies right away that the image of $\varphi$ cannot coincide with $\mathcal S$. This implies that $\dim (\mathcal M)=20$, that $\varphi$ is dominant and generically finite. We are left to prove that its degree is 1. 
We argue by contradiction and suppose that the degree is greater than 1. Let $V$ be a general member of $ \mathcal H$. Then there is another $V'\in \mathcal H$ such that $S(V)=S(V')$ as polarised surfaces. Thus, up to a projective transformation of $\mP^5$, we can assume that $S(V)$ coincides with $S(V')$ as surfaces in $\mP^5$. Consider then the pencil of sextic hypersurfaces generated by the two EPW--sextics $\Phi_V$ and $\Phi_{V'}$. All these sextics have $S(V)=S(V')$ as a double surface and are   EPW--sextics. But then the fibre of $\varphi$ passing through $V$ would have positive dimension, a contradiction.  \end{proof}

\end{document}